\documentclass{amsart}

\usepackage{graphicx}
\usepackage[utf8]{inputenc}
\usepackage[english]{babel}
\usepackage{amsmath}
\usepackage{amsfonts}
\usepackage{amssymb}
\usepackage{xcolor}
\usepackage{amsthm}
\usepackage[all]{xy}
\usepackage{enumitem}

\newtheorem{thm}{Theorem}[section]
\newtheorem{lemma}[thm]{Lemma}
\newtheorem{prop}[thm]{Proposition}
\newtheorem{cor}[thm]{Corollary}

\theoremstyle{definition}
\newtheorem{defn}[thm]{Definition}
\newtheorem{example}[thm]{Example}

\theoremstyle{remark}
\newtheorem{rem}[thm]{Remark}

\numberwithin{equation}{section}

\DeclareMathOperator{\ann}{Ann}

\DeclareMathOperator{\Ker}{Ker}

\DeclareMathOperator{\Hom}{Hom}
\DeclareMathOperator{\End}{End}

\DeclareMathOperator{\Rad}{Rad}

\newcommand{\sm}{\sigma[M]}

\newcommand{\dleq}{\leq^\oplus}

\begin{document}

\title{The nilpotency of the prime radical of a Goldie module}

\author{John A. Beachy}
\address{Department of Mathematical Sciences, Northern Illinois University,
Dekalb 60115, Illinois, US}
\email{jbeachy@niu.edu}

\author{Mauricio Medina-B\'arcenas}
\address{Facultad de Ciencias F\'isico-Matem\'aticas, Benem\'erita Universidad Aut\'onoma de Puebla, Av. San Claudio y 18 Sur, Col.San Manuel, Ciudad Universitaria, 72570, Puebla, M\'exico}
\email{mmedina@fcfm.buap.mx, corresponding author.}
\thanks{The second author was supported by the grant ``CONACYT-Estancias Posdoctorales 1er A\~no 2019 - 1''.}


\date{}



\begin{abstract}
With the notion of prime submodule defined by F. Raggi et.al. we prove that the intersection of all prime submodules of a Goldie module $M$, is a nilpotent submodule provided that $M$ is retractable and $M^{(\Lambda)}$-projective for every index set $\Lambda$. This extends the well known fact that in a left Goldie ring, the prime radical is nilpotent.
\end{abstract}

\maketitle

\section{Introduction}

For a left Artinian ring $R$, it is known that the prime radical of $R$ coincides with its Jacobson radical. Therefore, in any left Artinian ring the prime radical is always nilpotent. For an arbitrary ring $R$, it is no longer true that the prime radical equals the Jacobson radical. In general, the prime radical of $R$ is contained properly in its Jacobson radical. It can be seen that the prime radical of a ring $R$ consists of nilpotent elements, that is, the prime radical is a nil ideal. It has been a recurrent question in ring theory whether the nil (one-sided) ideals are nilpotent. Levitzki, in 1939 (see \cite[10.30]{lamfirst} and the paragraph below the proof) proved that every nil one-sided ideal of a left Noetherian ring is nilpotent. In particular, the prime radical of a left Noetherian ring is nilpotent. This result was extended for a left Goldie ring by Lanski \cite{lanski1969nil} in 1969. 

In \cite{BicanPr}, a product of modules generalizing the product of a left ideal with a left module  was introduced. Hence, this product gives a product of submodules of a given module and it coincides with the usual product of ideals when the module is the base ring. In 2002, the first author proved that this product in a left $R$-module $M$ is associative provided that $M$ is $M^{(\Lambda)}$-projective for every index set $\Lambda$ \cite[Proposition 5.6]{beachy2002m}. Later, in \cite{raggiprime} (resp. \cite{raggisemiprime}), it is introduced the concept of  prime (resp. semiprime) submodule of a module $M$ using the product mentioned above. As part of the second author's dissertation, in \cite{maugoldie} semiprime Goldie modules were presented as a generalization of semiprime left Goldie rings. In this paper, for a left $R$-module $M$ which is $M^{(\Lambda)}$-projective for every index set $\Lambda$, we consider the intersection of all prime submodules of $M$ and we prove that this intersection is nilpotent when $M$ is Goldie, extending Lanski's result \cite[Theorem 1]{lanski1969nil}. The nilpotency of the prime radical of a Goldie module will be used in a subsequent paper when the finite reduced rank is studied in categories of type $\sigma[M]$ \cite{beachyreduced}.

In order to achieve our goal, this paper is divided in five sections. The first section is this introduction and the next one compiles the background needed to make this work as self-contained as possible. In Section \ref{LocNilSub}, nil-submodules and locally nilpotent submodules are defined. It is proved that the sum of all locally nilpotent submodules of a module inherits the property (Lemma \ref{sumlocnil}). For a module $M$, this maximal locally nilpotent submodule is denoted by $\mathfrak{L}(M)$ and it happens that $\mathfrak{L}\left(M/\mathfrak{L}(M) \right)=0$ provided that $M$ is $M^{(\Lambda)}$-projective for every index set $\Lambda$ (Proposition \ref{Lfiyrad}). Moreover, the submodule $\mathfrak{L}(M)$ coincides with the intersection of all prime submodules of $M$ (Corollary \ref{nesL}). At the end of this section, using the operator $\mathfrak{L}(\_)$, a characterization of semiprime rings in terms of free modules is given (Corollary \ref{rsp}). In Section \ref{AccandNil}, we introduce the notion of ascending chain condition (acc) on annihilators on a module $M$. We study the behavior of nil submodules of a module satisfying acc on annihilators. We show that fully invariant nil submodules are locally nilpotent when $M$ is $M^{(\Lambda)}$-projective for every index set $\Lambda$ and satisfies acc on annihilators (Proposition \ref{accnillocnil}). Also in this section, a fundamental result of the paper is proved (Proposition \ref{subm}). Finally in Section \ref{Nilpotency}, the main result is proved. It is shown that a fully invariant nil submodule of a Goldie module $M$ is nilpotent, under the hypotheses that $M$ is $M^{(\Lambda)}$-projective for every index set $\Lambda$ and retractable (Theorem \ref{main}). As under the previous  hypothesis on $M$, when $M$ is a Goldie module its endomorphism ring is a right Goldie ring (Lemma \ref{mgolsgol}), the proof of Theorem \ref{main} is focused on finding elements in the endomorphism ring of a Goldie module like those in \cite[Lemma 8]{lanski1969nil}. 

Throughout this paper $R$ will denote an associative ring with unit and all the $R$-modules will be unitary left modules. The notation $N\leq M$ and $N<M$ will mean $N$ is a submodule of $M$ and $N$ is a proper submodule of $M$ respectively. Given an $R$-module $M$, the endomorphism ring of $M$ usually will be denoted by $S=\End_R(M)$. A direct sum of copies of a module $M$ will be denoted by $M^{(\Lambda)}$ where $\Lambda$ is a set. 

\section{Preliminaries}
Let $M$ and $N$ be left $R$-modules. It is said that $M$ is \emph{$N$-projective} if for any epimorphism $\rho:N\to L$ and any homomorphism $\alpha:M\to L$, there exists $\overline{\alpha}:M\to N$ such that $\rho\overline{\alpha}=\alpha$. It is true that if $M$ is $N_i$-projective for a finite family of modules $\{N_1,...,N_\ell\}$, then $M$ is $\bigoplus_{i=1}^\ell N_i$-projective \cite[18.2]{wisbauerfoundations}. In general, this is not true for arbitrary families. A module $M$ is called \emph{quasi-projective} if, $M$ is $M$-projective. In most of the results in this paper, we will assume that $M$ is $M^{(\Lambda)}$-projective for every index set $\Lambda$, this hypothesis is satisfied by every finitely generated quasi-projective module for example \cite[18.2]{wisbauerfoundations}. The condition that $M$ is $M^{(\Lambda)}$-projective for every index set $\Lambda$ is equivalent to say that $M$ is projective in the category $\sm$ \cite[18.3]{wisbauerfoundations}, where $\sm$ is the category consisting of all $M$-subgenerated modules. In many of our references, this last equivalence is used. It will also be used, inside the proofs, that if $M$ is $M^{(\Lambda)}$-projective for every index set $\Lambda$ and $N$ is a fully invariant submodule of $M$ then, $M/N$ is $(M/N)^{(\Lambda)}$-projective for every index set $\Lambda$ \cite[Lemma 9]{vanmodules}.


Given a left $R$-module $M$ and $N,K\leq M$, \emph{the product of $N$ with $K$ in $M$} is defined as 
\[N_MK=\sum\{f(N)\mid f\in\Hom_R(M,K)\}.\]
If $M$ is $M^{(\Lambda)}$-projective for every index set $\Lambda$, this product gives an associative operation in the lattice of submodules of $M$. Moreover, if $N$ and $K$ are fully invariant submodules of $M$, then $N_MK$ is fully invariant.  For $N\leq M$, \emph{the powers of $N$} are defined recursively as follows: $N^1=N$ and $N^{\ell+1}={N^\ell}_MN$. \footnote{Do not confuse the notation with the direct product of $\ell+1$ copies of $N$ which will be denoted by $N^{(\ell+1)}$.} It is said that a submodule $N\leq M$ is \emph{nilpotent} if $N^\ell=0$ for some $\ell>0$. Some general properties of this product are listed in \cite[Proposition 1.3]{PepeGab}. Also, we give the following useful lemmas.

\begin{lemma}\label{proddirsumm}
	Let $M$ be a module and $N\leq M$. Then $K_ML\subseteq K_NL$ for all $K,L\leq N$. In addition, if $N\dleq M$ then $K_ML=K_NL$.
\end{lemma}

\begin{proof}
	Let $K,L$ be submodules of $N$. Using the restriction of a homomorphism, if $f:M\to L$ is any homomorphism then $f|_N(K)=f(K)$ since $K\leq N$. Thus $K_ML\subseteq K_NL$. Now suppose $N$ is a direct summand of $M$. If $f:N\to L$ then $f\oplus 0:M\to L$. Hence $K_NL\subseteq K_ML$. Thus $K_NL=K_ML$.
\end{proof}

\begin{lemma}\label{fprod}
	Let $M$ be a quasi-projective module. If $A$ and $B$ are submodules of $M$ and $f\in\End_R(M)$, then $f(A_MB)=A_Mf(B)$ and $f(A)_MB\subseteq A_MB$.
\end{lemma}

\begin{proof}
	Let $f\left( \sum_{i=1}^\ell g_i(a_i)\right)\in f(A_MB)$. Then,
	\[f\left( \sum_{i=1}^\ell g_i(a_i)\right)=\sum_{i=1}^\ell fg_i(a)\in A_Mf(B).\]
	On the other hand, let $\sum_{i=1}^\ell h_i(a_i)\in A_Mf(B)$. The restriction $f|_B:B\to f(B)$ is an epimorphism. Since $M$ is quasi-projective, there exists $g_i:M\to B$ such that $fg_i=h_i$ for all $i$. Hence,
	\[\sum_{i=1}^\ell h_i(a_i)=\sum_{i=1}^\ell fg_i(a_i)=f\left( \sum_{i=1}^\ell g_i(a_i)\right)\in f(A_MB).\]
	For the other assertion, let $\sum_{i=1}^\ell g_i(f(a_i))\in f(A)_MB$. Then $g_if:M\to B$. Thus, $\sum_{i=1}^\ell g_if(a_i)\in A_MB$.
\end{proof}

\begin{lemma}\label{epiproduct}
	Let $M$ be quasi-projective and let $K,N\leq M$ be two submodules with $K$ fully invariant in $M$. If $\pi:M\to M/K$ is the canonical projection, then $\pi(N_MN)=\pi(N)_{M/K}\pi(N)$. 
\end{lemma}

\begin{proof}
	Since $K$ is fully invariant in $M$, given $f\in\Hom_R(M,N)$ there is an $\bar{f}\in\Hom_R(M/K,\pi(N))$ such that $\pi f=\bar{f}\pi$. Now, since $M$ is quasi-projective, if $\bar{f}\in\Hom_R(M/K,\pi(N))$ then there exists $f\in\Hom_R(M,N)$ such that $\pi f=\bar{f}\pi$. Thus, there is an epimorphism $\Hom_R(M,N)\to \Hom_R(M/K,\pi(N))\to 0$ induced by $\pi$.
	It follows that:
	\begin{equation*}
	\begin{split}
	\pi(N_MN) & =\pi\left( \sum\{f(N)\mid f\in\Hom_R(M,N)\}\right) \\
	& =\sum\{\pi f(N)\mid f\in \Hom_R(M,N)\}  \\
	& =\sum\{\bar{f}\pi(N)\mid \bar{f}\in\Hom_R(M/K,\pi(N))\}\\
	& =\pi(N)_{M/K}\pi(N).
	\end{split}
	\end{equation*}
\end{proof}

In a natural way, a fully invariant submodule $P$ of a module $M$ is said to be a \emph{prime submodule} (resp. \emph{semiprime submodule}) if $N_MK\leq P$ (resp. $N_MN\leq P$) implies $N\leq P$ or $K\leq P$ (resp. $N\leq P$) for every fully invariant submodules $N$ and $K$. It is said that $M$ is a \emph{prime module} (resp. \emph{semiprime module}) if $0$ is a prime (resp. semiprime) submodule. Many properties known for prime or semiprime ideals can be extended to prime or semiprime sumodules as it has been done in  \cite{beachy2002m,beachy2020fully,PepeGab,mauozcan,raggiprime,raggisemiprime}. Recall that a ring $R$ is \emph{left Goldie}, if $R$ satisfies the ascending chain condition (acc) on left annihilators and has finite uniform dimension. This notion can be extended to left modules as follows: a left module $M$ is called \emph{Goldie} if the set $\{\bigcap_{f\in X}\Ker f\mid X\subseteq \End_R(M)\}$ satisfies acc and $M$ has finite uniform dimension. It is clear that a ring $R$ is a left Goldie ring if and only if $_RR$ is a Goldie module \cite{maugoldie}.

\section{Locally nilpotent submodules}\label{LocNilSub}

%
%
%
%
%
%

\begin{defn}
Let $M$ be an $R$-module and $S=\End_R(M)$. A submodule $N$ of $M$ is called a \emph{nil-submodule} if for all $n\in N$, the right ideal $\Hom_R(M,Rn)$ of $S$ is nil. That is, if each $f:M\to Rn$ is nilpotent.
\end{defn}

%
%
%

\begin{lemma}\label{factornil}
Let $M$ be a quasi-projective module and let $N$ be a nil-submodule of $M$. Then, $\frac{N+K}{K}$ is a nil-submodule of $M/K$ for all $K\leq M$. 
\end{lemma}

\begin{proof}
Let $n+K\in \frac{N+K}{K}$ and $f\in\Hom_R(M/K,R(n+K))$. Let $\pi:M\to M/K$ be the canonical projection and consider $\pi|:Rn\to R(n+K)$ the restriction to $Rn$. Since $M$ is quasi-projective, there exists $g:M\to Rn$ such that $(\pi|)g=f\pi$. By hypothesis, there exists $k>1$ such that $g^k=0$. Therefore,
\[0=(\pi|)g^k=f\pi g^{k-1}=f(\pi g)g^{k-2}=f(f\pi)g^{k-2}=\cdots=f^k\pi.\]
This implies that $f^k=0$ and hence $\frac{N+K}{K}$ is a nil-submodule of $M/K$.
\end{proof}


Recall that a subset $T$ of a ring $R$ is called \emph{locally nilpotent} if, for any finite subset $\{t_1,...,t_n\}\subseteq T$ there exists an integer $\ell$ such that any product of $\ell$ elements from $\{t_1,...,t_n\}$ is zero \cite[pp. 166]{lamfirst}.

\begin{defn}
Let $N$ be a submodule of a module $M$. It is said that $N$ is \emph{locally nilpotent} if for any subset $\{n_1,...,n_k\}\subseteq N$ there exists an integer $\ell>0$ such that any combination of $\ell$ elements of $\{n_1,...,n_k\}$ satisfies that 
\[{Rn_{i_1}}_M{Rn_{i_2}}_M\cdots{_MRn_{i_\ell}}=0.\]
\end{defn}

Note that if $N$ is a locally nilpotent submodule of $M$, then $Rn$ is a nilpotent submodule for every $n\in N$, but $N$ might not be nilpotent as the following example shows.

\begin{example}
Consider the $\mathbb{Z}$-module $\mathbb{Q}$. Every nonzero cyclic submodule of $_\mathbb{Z}\mathbb{Q}$ is isomorphic to $\mathbb{Z}$ and $\mathbb{Z}_{\mathbb{Q}}\mathbb{Z}=0$. This implies that $_\mathbb{Z}\mathbb{Q}$ is locally nilpotent. But $\mathbb{Q}^\ell=\mathbb{Q}$ for all $\ell>0$.
\end{example}

\begin{rem}
It is clear that if a left ideal $I$ of a ring $R$ is a locally nilpotent submodule, then $I$ is a locally nilpotent subset. However, the converse might not be true. For, consider the ring $R$ given in \cite[pp. 167]{lamfirst}. The ring $R$ is prime, so has no nonzero nilpotent submodules. But it is proved that $R$ has a cyclic left ideal $Rx$ which is locally nilpotent. Therefore, $Rx$ cannot be a locally nilpotent submodule of $_RR$.
\end{rem}

\begin{rem}
If $N\leq M$ is locally nilpotent, then $N$ is a nil-submodule. For, let $n\in N$ and $f\in\Hom_R(M,Rn)$. Consider $\{n\}\subseteq N$. Then, there is an $\ell>0$ such that $(Rn)^\ell=0$. Since $f(n)\in Rn_MRn$, $f^{\ell-1}(n)\in(Rn)^{\ell}=0$. Thus,  $f^\ell=0$. On the other hand, if $N\leq M$ is a nilpotent submodule, then $N$ is locally nilpotent.
\end{rem}

\begin{lemma}\label{fgnilp}
Suppose $M$ is $M^{(\Lambda)}$-projective for every index set $\Lambda$ and let $N$ be a finitely generated submodule of $M$. If $N$ is locally nilpotent, then $N$ is nilpotent.
\end{lemma}

\begin{proof}
We proceed by induction on the number of generators of $N$. Suppose $N=Rn$. Since $N$ is locally nilpotent,  there exists $\ell>0$ such that the product $\ell$ times of $Rn$  is zero, that is, $N^\ell=(Rn)^\ell=0$. Now suppose that $N=Rn_1+\cdots+Rn_k$ and that the result is valid for any locally nilpotent finitely generated submodule with less than $k$ generators. Put $K=Rn_1+\cdots+Rn_{k-1}$. By induction hypothesis there exist $\ell_1,\ell_2>0$ such that $K^{\ell_1}=0$ and $(Rn_k)^{\ell_2}=0$. Let $\ell>\ell_1+\ell_2$. Then,
\begin{equation}\label{a}
\begin{split}
N^\ell & =(K+Rn_k)^\ell\\
 & =K^\ell+\cdots+{K^{\alpha_1}}_M{(Rn_k)^{\beta_1}}_M\cdots _M{K^{\alpha_h}}_M(Rn_k)^{\beta_h}+\cdots +(Rn_k)^\ell
\end{split}
\end{equation}
where $\alpha_i,\beta_i\geq 0$, $\sum_{i=1}^h\alpha_i+\beta_i=\ell$ and with the convention ${A^0}_MB=B$ and $A_MB^0=A$. Then $\sum_{i=1}^h\alpha_i>\ell_1$ or $\sum_{i=1}^h\beta_i>\ell_2$ since $\ell>\ell_1+\ell_2$. Suppose $\beta=\sum_{i=1}^h\beta_i>\ell_2$. Consider the summand $X={K^{\alpha_1}}_M{(Rn_k)^{\beta_1}}_M\cdots _M{K^{\alpha_h}}_M(Rn_k)^{\beta_h}$ in the expression (\ref{a}). If $\beta_h\neq 0$, then ${K^{\alpha_h}}_M(Rn_k)^{\beta_h}\subseteq (Rn_k)^{\beta_h}$. Hence 
\begin{equation*}
\begin{split}
{K^{\alpha_1}}_M{(Rn_k)^{\beta_1}}_M\cdots _M{K^{\alpha_h}}_M(Rn_k)^{\beta_h} & \subseteq {K^{\alpha_1}}_M{(Rn_k)^{\beta_1}}_M\cdots _M{K^{\alpha_{h-1}}}_M(Rn_k)^{\beta_{h-1}+\beta_h} \\
& \cdots \\
& \subseteq {K^{\alpha_1}}_M(Rn_k)^\beta.
\end{split}
\end{equation*}
Hence $X=0$. If $\beta_h=0$, we have $X={K^{\alpha_1}}_M{(Rn_k)^{\beta_1}}_M\cdots _M{(Rn_k)^{\beta_{h-1}}}_M{K^{\alpha_h}}$. Proceeding as above, $X\subseteq {K^{\alpha_1}}_M{(Rn_k)^\beta}_MK^{\alpha_h}$. Thus, $X=0$. Analogously if $\alpha=\sum_{i=1}^h\alpha_i>\ell_1$. Hence $N^\ell=0$.
\end{proof}

\begin{rem}
	Notice that the proof of Lemma \ref{fgnilp} implies that in a module $M$ which is $M^{(\Lambda)}$-projective for every index set $\Lambda$, the finite sum of nilpotent submodules is nilpotent.
\end{rem}

\begin{lemma}\label{sumlocnil}
Suppose $M$ is $M^{(\Lambda)}$-projective for every index set $\Lambda$. If $N$ and $L$ are locally nilpotent submodules of $M$, then $N+L$ is locally nilpotent.
\end{lemma}

\begin{proof}
Let $\{n_1+l_1,..,n_k+l_k\}\subseteq N+L$. For $\{n_1,...,n_k\}\subseteq N$ there exists $\ell_N>0$ such that any combination of $\ell_N$ elements of $\{n_1,...,n_k\}$ satisfies that ${Rn_{i_1}}_M{Rn_{i_2}}_M\cdots{_MRn_{i_{\ell_N}}}=0$. Analogously, for the subset $\{l_1,...,l_k\}\subseteq  L$, there exists $\ell_L>0$ such that any combination of $\ell_L$ elements of $\{l_1,...,l_k\}$ satisfies that ${Rl_{i_1}}_M{Rl_{i_2}}_M\cdots{_MRl_{i_{\ell_L}}}=0$. Take $\ell>\ell_N+\ell_L$. Then 
\[{R(n_{i_1}+l_{i_1})}_M\cdots{_MR(n_{i_\ell}+l_{i_\ell})}\subseteq\left( {Rn_{i_1}+Rl_{i_1}}\right) _M\cdots _M\left(Rn_{i_\ell}+Rl_{i_\ell}\right)\]
\[= \cdots +{(Rn_{i_1})^{\alpha_1}}_M{(Rl_{i_1})^{\beta_1}}_M{(Rn_{i_2})^{\alpha_2}}_M{(Rl_{i_2})^{\beta_2}}_M\cdots_M{(Rn_{i_\ell})^{\alpha_\ell}}_M{(Rl_{i_\ell})^{\beta_\ell}}+\cdots\]
where $\alpha_j,\beta_j\in\{0,1\}$ with the convention that ${(Rn)^0}_MRl=Rl$ and $Rn_M(Rl)^0=Rn$. Also, these exponents satisfy $\sum_{j=1}^\ell \alpha_j+\beta_j=\ell$.
Let us focus on the summand
\[X={(Rn_{i_1})^{\alpha_1}}_M{(Rl_{i_1})^{\beta_1}}_M{(Rn_{i_2})^{\alpha_2}}_M{(Rl_{i_2})^{\beta_2}}_M\cdots_M{(Rn_{i_\ell})^{\alpha_\ell}}_M{(Rl_{i_\ell})^{\beta_\ell}}\]
Suppose that $\sum_{j=1}^\ell\alpha_j\geq \ell_N$. Since the product ${(Rn_{k-1})^{\beta_{k-1}}}_M(Rl_{k})^{\alpha_{k}}\subseteq (Rl_{k})^{\alpha_{k}}$, 
\[X\subseteq\left( {(Rn_{i_1})^{\alpha_1}}_M{(Rn_{i_2})^{\alpha_2}}_M\cdots_M{(Rn_{i_{\ell_N}})^{\alpha_{\ell_N}}}\right) _M{(Rl_{i_{\ell_N}})^{\beta_{\ell_N}}}_M\cdots_M{(Rn_{i_\ell})^{\alpha_\ell}}_M{(Rl_{i_\ell})^{\beta_\ell}}.\]
Then $X=0$. Analogously, if $\sum_{j=1}^\ell\beta_j\geq \ell_L$, then $X=0$. Because of the choice of $\ell$ it cannot be possible that $\sum_{j=1}^\ell\alpha_j< \ell_N$ and $\sum_{j=1}^\ell\beta_j< \ell_L$. Thus, $N+L$ is locally nilpotent.
\end{proof}

%

\begin{prop}\label{Lfiyrad}
Suppose $M$ is $M^{(\Lambda)}$-projective for every index set $\Lambda$ and let $\mathfrak{L}(M)$ be the sum of all locally nilpotent submodules of $M$. Then $\mathfrak{L}(M)$ is a fully invariant locally nilpotent submodule. Moreover $\mathfrak{L}(M/\mathfrak{L}(M))=0$.
\end{prop}

\begin{proof}
It follows from Lemma \ref{sumlocnil} that $\mathfrak{L}(M)$ is locally nilpotent. Let $f\in\End_R(M)$ be any endomorphism and let $\{f(m_1),...,f(m_k)\}$ be any subset of $f(\mathfrak{L}(M))$. Since $\mathfrak{L}(M)$ is locally nilpotent and $\{m_1,...,m_k\}\subseteq \mathfrak{L}(M)$, there exists $\ell>0$ such that any combination of $\ell$ elements of $\{m_1,...,m_k\}$ satisfies that ${Rm_{i_1}}_M{Rm_{i_2}}_M\cdots{_MRm_{i_{\ell}}}=0$. Then, by Lemma \ref{fprod},
\[{Rf(m_{i_1})}_M\left( {Rf(m_{i_2})} _M\cdots{_MRf(m_{i_{\ell}})}_MRf(m_{i_{\ell+1}})\right) \]
\[\subseteq {Rm_{i_1}}_M\left( {Rf(m_{i_2})} _M\cdots{_MRf(m_{i_{\ell}})}_MRf(m_{i_{\ell+1}})\right)\]
\[\subseteq {Rm_{i_1}}_M {Rm_{i_2}} _M\cdots{_MRf(m_{i_{\ell}})}_MRf(m_{i_{\ell+1}})\]
\[\subseteq\cdots\subseteq {Rm_{i_1}}_M{Rm_{i_2}} _M\cdots{_MRm_{i_{\ell}}} _MRf(m_{i_{\ell+1}})=0.\]

Therefore, $f(\mathfrak{L}(M))\subseteq \mathfrak{L}(M)$. Thus, $\mathfrak{L}(M)$ is fully invariant. Now, let $\pi:M\to M/\mathfrak{L}(M)$ be the canonical projection and $x+\mathfrak{L}(M)\in M/\mathfrak{L}(M)$. Consider $R(x+\mathfrak{L}(M))=\pi(Rx)$. It follows from Lemma \ref{epiproduct}, that $\mathfrak{L}(M/\mathfrak{L}(M))=0$.
\end{proof}

\begin{cor}\label{Lsp}
Suppose $M$ is $M^{(\Lambda)}$-projective for every index set $\Lambda$. Then, $\mathfrak{L}(M)$ is a semiprime submodule of $M$.
\end{cor}

\begin{proof}
Let $K/\mathfrak{L}(M)$ be a submodule of $M/\mathfrak{L}(M)$ such that $\left(K/\mathfrak{L}(M) \right) ^2=0$. Hence $K/\mathfrak{L}(M)$ is nilpotent and so $K/\mathfrak{L}(M)\subseteq \mathfrak{L}(M/\mathfrak{L}(M))=0$. Thus, $M/\mathfrak{L}(M)$ is a semiprime module which implies that $\mathfrak{L}(M)$ is a semiprime submodule of $M$ by \cite[Proposition 13]{raggisemiprime}.
\end{proof}


\begin{defn}
Let $M$ be a module. The \emph{prime radical} of $M$ is defined as the intersection of all prime submodules of $M$. 
\end{defn}

If $M$ is $M^{(\Lambda)}$-projective for every index set $\Lambda$, in \cite[22.3]{wisbauerfoundations} it is shown that $\Rad(M)\neq M$, that is, $M$ has maximal submodules. By \cite[Corollary 4.11 and Example 4.14]{medinageneralization} each maximal submodule of $M$ contains a prime submodule, hence $Spec(M)\neq\emptyset$ and so the prime radical of $M$ is a proper submodule.

\begin{cor}\label{nesL}
Suppose $M$ is $M^{(\Lambda)}$-projective for every index set $\Lambda$. Let $N$ be the prime radical of $M$. Then,  $N=\mathfrak{L}(M)$.
\end{cor}

\begin{proof}
It follows from Corollary \ref{Lsp}, that $\mathfrak{L}(M)$ is a semiprime submodule. Hence, $N\subseteq \mathfrak{L}(M)$ by \cite[Proposition 1.11]{maugoldie}. Now, let $l\in \mathfrak{L}(M)$. Since $\mathfrak{L}(M)$ is a locally nilpotent submodule, $Rl$ is nilpotent. Since $N$ is a semiprime submodule, $l\in N$. Thus, $\mathfrak{L}(M)\subseteq N$.
\end{proof}


\begin{cor}\label{prnilnet}
Let $M$ be a quasi-projective module. If $M$ is Noetherian, then the prime radical of $M$ is nilpotent.
\end{cor}

\begin{proof}
It follows from Lemma \ref{fgnilp} and Corollary \ref{nesL}.
\end{proof}

\begin{example}
	Let $n,p$ be integers with $p$ prime. Then, the $\mathbb{Z}$-module $\mathbb{Z}/p^n\mathbb{Z}$ is quasi-projective. It follows from Corollary \ref{prnilnet} that $\mathfrak{L}(\mathbb{Z}/p^n\mathbb{Z})=\mathbb{Z}/p^{n-1}\mathbb{Z}$ is a nilpotent submodule.
\end{example}

\begin{lemma}\label{Lsumas}
Suppose $M$ is $M^{(\Lambda)}$-projective for every index set $\Lambda$. If $M=\bigoplus_{i\in I} M_i$, then $\mathfrak{L}(M)=\bigoplus_{i\in I}\mathfrak{L}(M_i)$.
\end{lemma}

\begin{proof}
It follows from Lemma \ref{proddirsumm} that $\mathfrak{L}(M_i)\subseteq \mathfrak{L}(M)$ for all $i\in I$. Let $\pi_j:M\to M_j$ be the canonical projection. Since $M_j\leq M$, we can see $\pi$ as an endomorphism of $M$. Hence $\pi(\mathfrak{L}(M))\subseteq M_j\cap \mathfrak{L}(M)$ by Proposition \ref{Lfiyrad}. Now, let $X\leq M$ be a locally nilpotent submodule such that $X\leq M_j$. Since $M_j$ is a direct summand of $M$, $X\leq \mathfrak{L}(M_j)$ by Lemma \ref{proddirsumm}. Therefore $\pi(\mathfrak{L}(M))\leq \mathfrak{L}(M_j)$. Hence $\pi(\mathfrak{L}(M))\leq \mathfrak{L}(M_i)$ for all $i\in I$. It follows that $\mathfrak{L}(M)\subseteq \bigoplus_{i\in I} \mathfrak{L}(M_i)$. Thus, $\mathfrak{L}(M)=\bigoplus_{i\in I} \mathfrak{L}(M_i)$.
\end{proof}

\begin{prop}
Suppose $M$ is $M^{(\Lambda)}$-projective for every index set $\Lambda$ and that $M=\bigoplus_{i\in I}M_i$. Then, $M$ is semiprime if and only if $M_i$ is a semiprime module for all $i\in I$.
\end{prop}

\begin{proof}
It follows from Lemma \ref{Lsumas} that 
\[\mathfrak{L}(M)=\mathfrak{L}\left(\bigoplus_{i\in I} M_i \right)=\bigoplus_{i\in I} \mathfrak{L}(M_i).\]
Thus, $\mathfrak{L}(M)=0$ if and only if $\mathfrak{L}(M_i)=0$ for all $i\in I$.
\end{proof}

\begin{cor}\label{rsp}
The following conditions are equivalent for a ring $R$.
\begin{enumerate}[label=\emph{(\alph*)}]
\item $R$ is a semiprime ring.
\item Every free left (right) $R$-module is semiprime.
\item Every projective left (right) $R$-module is semiprime.
\end{enumerate}
\end{cor}

\begin{cor}
	The following conditions are equivalent for a ring $R$:
	\begin{enumerate}[label=\emph{(\alph*)}]
		\item The prime radical of $R$ is nilpotent.
		\item The prime radical of every left (right) free $R$-module is nilpotent.
		\item The prime radical of every left (right) projective $R$-module is nilpotent.
	\end{enumerate}
\end{cor}

\begin{proof}
	(a)$\Rightarrow$(b) Consider a free $R$-module $R^{(\Lambda)}$. Then $\mathfrak{L}\left(R^{(\Lambda)} \right)=\left(\mathfrak{L}(R) \right)^{(\Lambda)}$ by Lemma \ref{Lsumas}. By hypothesis, there exists $n>0$ such that $\mathfrak{L}(R)^n=0$. Suppose $n=2$. Using Lemma \ref{proddirsumm} and the properties of the product we get
	\[{\left(\mathfrak{L}(R) \right)^{(\Lambda)}}_{R^{(\Lambda)}}\left(\mathfrak{L}(R) \right)^{(\Lambda)}=\left[ {\left(\mathfrak{L}(R) \right)^{(\Lambda)}}_{R^{(\Lambda)}}\mathfrak{L}(R)\right] ^{(\Lambda)}=\left[\sum_{\Lambda}\left(\mathfrak{L}(R) _{R^{(\Lambda)}}\mathfrak{L}(R) \right)\right]^{(\Lambda)}\]
	\[=\left[\sum_{\Lambda}\left(\mathfrak{L}(R) _{R}\mathfrak{L}(R) \right)\right]^{(\Lambda)}=\left[\sum_{\Lambda}\left(\mathfrak{L}(R)^2\right)\right]^{(\Lambda)}=0.\]
	Thus, 
	\[\left(\mathfrak{L}\left(R^{(\Lambda)}\right) \right)^n=\left[\sum_{\Lambda^{n-1}}\left(\mathfrak{L}(R)^n\right)\right]^{(\Lambda)}=0.\]
	Therefore, the prime radical of $R^{(\Lambda)}$ is nilpotent.
	
	(b)$\Rightarrow$(c) Let $P$ be a projective $R$-module. Then there exists a free $R$-module $R^{(\Lambda)}$ such that $R^{(\Lambda)}=P\oplus P'$. Hence $\mathfrak{L}\left( R^{(\Lambda)}\right) =\mathfrak{L}(P)\oplus\mathfrak{L}(P')$. By hypothesis, there exists $n>0$ such that $\mathfrak{L}\left( R^{(\Lambda)}\right)^n=0$. Therefore $\left(\mathfrak{L}(P)\oplus\mathfrak{L}(P') \right)^n=0$. Suppose $n=2$. Then
	\[0=\left(\mathfrak{L}(P)\oplus\mathfrak{L}(P') \right)^2=\left[\mathfrak{L}(P)\oplus\mathfrak{L}(P')\right]_{R^{(\Lambda)}}\left[\mathfrak{L}(P)\oplus\mathfrak{L}(P')\right]\]
	\[=\left( \left[\mathfrak{L}(P)\oplus\mathfrak{L}(P')\right]_{R^{(\Lambda)}}\mathfrak{L}(P)\right) \oplus\left( \left[\mathfrak{L}(P) \oplus \mathfrak{L}(P')\right]_{R^{(\Lambda)}}\mathfrak{L}(P')\right).\]
	On the other hand,
	\[\mathfrak{L}(P)_P\mathfrak{L}(P)=\mathfrak{L}(P)_{R^{(\Lambda)}}\mathfrak{L}(P)\subseteq\left[\mathfrak{L}(P)\oplus\mathfrak{L}(P')\right]_{R^{(\Lambda)}}\mathfrak{L}(P)=0.\]
	This implies that $\mathfrak{L}(P)$ is a nilpotent submodule of $P$.
	
	(c)$\Rightarrow$(a) is trivial.
\end{proof}

\section{Nil-submodules and acc on annihilators}\label{AccandNil}

For a ring $R$ and $Y\subseteq R$, the left (resp. right) annihilator of $Y$ in $R$ is denoted by $\ann_R^\ell(Y)$ (resp. $\ann_R^r(Y)$). For a subset $X\subseteq M$ of an $R$-module $M$ with $S=\End_R(M)$, the left annihilator of $X$ in $S$ is denoted by $\mathbf{l}_S(X)=\{f\in S\mid f(X)=0\}$. If $\left\langle X\right\rangle $ is the $R$-submodule generated by $X\subseteq M$ then $\mathbf{l}_S(X)=\mathbf{l}_S\left( \left\langle X\right\rangle \right)$. On the other hand, an $R$-module $M$ is said to have the \emph{acc on annihilators} if any chain
\[\bigcap_{f\in Y_1}\Ker f\subseteq\bigcap_{f\in Y_2}\Ker f\subseteq\cdots\]
with $Y_i\subseteq \End_R(M)$ ($i>0$) becomes stationary in finitely many steps.  

\begin{prop}\label{maccsacc}
If $M$ satisfies acc on annihilators, then $S=\End_R(M)$ satisfies acc on right annihilators.
\end{prop}

\begin{proof}
We will show that $S$ satisfies the descending chain condition (dcc) on left annihilators. Let 
\[\ann_S^\ell(X_1)\supseteq \ann_S^\ell(X_2)\supseteq\cdots\]
be a descending chain of left annihilators in $S$. Then, there is an ascending chain
\[\bigcap_{f\in\ann_S^\ell(X_1)}\Ker f\subseteq \bigcap_{f\in\ann_S^\ell(X_2)}\Ker f\subseteq\cdots\]
in $M$. By hypothesis, there exists a positive integer $k$ such that $\bigcap_{f\in\ann_S^\ell(X_k)}\Ker f= \bigcap_{f\in\ann_S^\ell(X_{k+i})}\Ker f$ for all $i>0$. We claim that 
\[\mathbf{l}_S\left( \bigcap_{f\in\ann_S^\ell(X_k)}\Ker f\right) =\ann_S^\ell(X_1).\]
Let $f\in\ann_S^\ell(X_k)$. It follows that $f\left( \bigcap_{f\in\ann_S^\ell(X_k)}\Ker f\right)=0$. Hence, $\ann_S^\ell(X_k)\subseteq\mathbf{l}_S\left( \bigcap_{f\in\ann_S^\ell(X_k)}\Ker f\right)$. Now, let $g\in\mathbf{l}_S\left( \bigcap_{f\in\ann_S^\ell(X_k)}\Ker f\right)$ and $h\in X_k$. It follows that $fh=0$ for all $f\in\ann_S^\ell(X_k)$, i.e., $h(M)\subseteq\bigcap_{f\in\ann_S^\ell(X_k)}\Ker f$. Therefore, $gh=0$. Thus, $g\in\ann_S^\ell(X_k)$. Then,
\[\ann_S^\ell(X_k)=\mathbf{l}_S\left( \bigcap_{f\in\ann_S^\ell(X_k)}\Ker f\right) = \mathbf{l}_S\left( \bigcap_{f\in\ann_S^\ell(X_{k+i})}\Ker f\right) =\ann_S^\ell(X_{k+i})\]
for all $i>0$. This implies that $S$ satisfies acc on right annihilators.
\end{proof}

Notice that the converse of Proposition \ref{maccsacc} is not true in general, as the next example shows.

\begin{example}
	Consider the $\mathbb{Z}$-module $M=\mathbb{Z}_{p^\infty}$ with $p$ a prime number. It is known that $S=\End_\mathbb{Z}(M)$ is isomorphic to the ring of $p$-adic numbers. Thus, $S$ is a commutative Noetherian ring. On the other hand, each submodule of $M$ is the kernel of an endomorphism of $M$, hence $M$ does not satisfy the acc on annihilators.
\end{example}


\begin{lemma}\label{nilpsubnil}
Let $M$ be $M^{(\Lambda)}$-projective for every index set $\Lambda$ and let $K<N$ fully invariant submodules of $M$ with $N$ a nil-submodule. 
If $M$ satisfies acc on annihilators, then $N/K$ contains a nonzero nilpotent submodule.
\end{lemma}

\begin{proof}
If $\Hom_R(M/K,N/K)=0$, then $N/K$ is nilpotent. Also, if $\Hom(M,N)=0$, then $\Hom_R(M/K,N/K)=0$ by the projectivity condition. Hence, we can assume that $\Hom_R(M/K,N/K)\neq0$ and $\Hom_R(M,N)\neq 0$. Given $f\in\Hom_R(M,N)$, since $K$ is fully invariant, there is a homomorphism $\overline{f}:M/K\to N/K$. On the other hand, since $M$ is quasi-projective, for any $g\in\Hom_R(M/K,N/K)$ there exists $f\in\Hom_R(M,N)$ such that $\pi f=g\pi$ where $\pi:M\to M/K$ is the canonical projection. Hence, there is a surjective homomorphism of rings without one $\Hom_R(M,N)\to \Hom_R(M/K,N/K)$. By hypothesis and Lemma \ref{factornil} we have that, $\Hom_R(M,N)$ and $\Hom_R(M/K,N/K)$ are nil-ideals of $\End_R(M)$ and $\End_R(M/K)$ respectively. It follows from Proposition \ref{maccsacc} and \cite[Lemma 1]{herstein1964nil} that $\Hom_R(M/K,N/K)$ contains a nilpotent ideal. In particular, there is $g\in\Hom_R(M/K,N/K)$ such that $gT$ is nilpotent, say $(gT)^n=0$, where $T=\End_R(M/K)$. It follows from \cite[18.4]{wisbauerfoundations} that $gT=\Hom_R(M/K,gTM/K)$. Therefore,
\begin{equation*}
\begin{split}
(gT(M/K))^{n+1} & ={gT(M/K)^n}_{M/K}gT(M/K) \\
& =\Hom_R(M/K,gT(M/K))gT(M/K)^n \\
& =\Hom_R(M/K,gT(M/K))^2gT(M/K)^{n-1} \\
& =\cdots \\
& =\Hom_R(M/K,gT(M/K))^ngT(M/K)=0.
\end{split}
\end{equation*}
\end{proof}

\begin{prop}\label{accnillocnil}
Let $M$ be $M^{(\Lambda)}$-projective for every index set $\Lambda$ and let $N$ be a fully invariant 
nil-submodule of $M$. If $M$ 
satisfies acc on annihilators, then $N$ is locally nilpotent.
\end{prop}

\begin{proof}
Since $\mathfrak{L}(M)$ is a semiprime (fully invariant) submodule of $M$, $M/\mathfrak{L}(M)$ is projective in $\sigma[M/\mathfrak{L}(M)]$ and $M/\mathfrak{L}(M)$ is retractable because it is a semiprime module. By Lemma \ref{factornil}, $\frac{N+\mathfrak{L}(M)}{\mathfrak{L}(M)}$ is a nil-submodule of $M/\mathfrak{L}(M)$. Suppose $N\nsubseteq \mathfrak{L}(M)$. By Lemma \ref{nilpsubnil}, $\frac{N+\mathfrak{L}(M)}{\mathfrak{L}(M)}$ contains a nonzero nilpotent submodule $\frac{K}{\mathfrak{L}(M)}$. Hence, there exists $\ell>0$ such that $K^\ell\subseteq \mathfrak{L}(M)$ by Lemma \ref{epiproduct}. It follows from Corollary \ref{Lsp} that $K\subseteq \mathfrak{L}(M)$. Hence $\frac{K}{\mathfrak{L}(M)}=0$. Thus, $\frac{N+\mathfrak{L}(M)}{\mathfrak{L}(M)}=0$, that is, $N\subseteq \mathfrak{L}(M)$.
\end{proof}

\begin{defn}
Let $M$ be a module and let $N\leq M$. The \emph{annihilator of $N$ in $M$} is the submodule
\[\ann_M(N)=\bigcap\{\Ker f\mid f\in\Hom_R(M,N)\}.\] 

The \emph{right annihilator of $N$ in $M$} is given by the submodule
\[\ann^r_M(N)=\sum\{K\leq M\mid N_MK=0\}.\]
\end{defn}

\begin{rem}
It is not difficult to see that $\ann_M(N)$ is a fully invariant submodule of $M$ and it is the largest submodule of $M$ such that $\ann_M(N)_MN=0$. When $M$ is $M^{(\Lambda)}$-projective for every index set $\Lambda$, $\ann^r_M(N)$ is fully invariant and is the largest submodule of $M$ such that $N_M\ann^r_M(N)=0$ \cite[Remark 1.15]{mauacc}.
\end{rem}

\begin{lemma}\label{rannintersection}
Let $M$ be a module and let $\{N_i\}_I$ be a family of submodules of $M$. Then $\bigcap_{i\in I}\ann^r_M(N_i)=\ann^r_M(\sum_{i\in I}N_i)$.
\end{lemma}

\begin{proof}
We always have that $\ann^r_M(\sum_{i\in I}N_i)\subseteq \ann^r_M(N_i)$ for all $i\in I$. Hence $\ann^r_M(\sum_{i\in I}N_i)\subseteq\bigcap_{i\in I} \ann^r_M(N_i)$. On the other hand, 
\[\left( \sum_{i\in I}N_i\right) _M\left( \bigcap_{i\in I}\ann^r_M(N_i)\right)=\sum_{i\in I}\left( {N_i}_M\bigcap_{i\in I}\ann^r_M(N_i)\right) =0.\]
Thus, $\bigcap_{i\in I}\ann^r_M(N_i)=\ann^r_M(\sum_{i\in I}N_i)$.
\end{proof}

\begin{lemma}\label{dccannr}
Let $M$ be a module. If $M$ satisfies acc on annihilators, then the set $\{\ann^r_M(N)\mid N\leq M\}$ satisfies dcc.
\end{lemma}

\begin{proof}
We just have to notice that $\ann^r_M(\ann_M(\ann^r_M(N)))=\ann^r_M(N)$.
\end{proof}

\begin{lemma}\label{dccl}
Suppose $M$ is $M^{(\Lambda)}$-projective for every index set $\Lambda$. Then,
\begin{enumerate}
\item $\mathbf{l}_S\left( \bigcap_{f\in\mathbf{l}_S(X)}\Ker f\right)=\mathbf{l}_S(X)$ for any subset $X$ of $M$.
\item $\bigcap_{f\in Y}\Ker f=\bigcap\left\lbrace \Ker g\mid g\in\mathbf{l}_S\left( \bigcap_{f\in Y}\Ker f\right) \right\rbrace $ for any subset $Y$ of $\End_R(M)$.
\item $M$ satisfies acc on annihilators if and only if the set $\{\mathbf{l}_S(X)\mid X\subseteq M\}$ satisfies dcc.
\end{enumerate}
\end{lemma}

\begin{proof}

\textit{(1)} Since $X\subseteq\bigcap_{f\in\mathbf{l}_S(X)}\Ker f$, we have that $\mathbf{l}_S\left( \bigcap_{f\in\mathbf{l}_S(X)}\Ker f\right)\subseteq\mathbf{l}_S(X)$. The other inclusion is obvious.

\textit{(2)} It is clear that $\bigcap_{f\in Y}\Ker f\subseteq\bigcap\left\lbrace \Ker g\mid g\in\mathbf{l}_S\left( \bigcap_{f\in Y}\Ker f\right) \right\rbrace$. On the other hand, since $Y\subseteq\mathbf{l}_S\left( \bigcap_{f\in Y}\Ker f\right)$, we have the other inclusion.

\textit{(3)} It is clear from \textit{(1)} and \textit{(2)}.
\end{proof}

%

\begin{prop}\label{factorrightacc}
Suppose $M$ is $M^{(\Lambda)}$-projective for every index set $\Lambda$ and that $M$ satisfies acc on annihilators. Let $\ann^r_M(N)$ be a right annihilator in $M$. Then $\overline{M}=M/\ann^r_M(N)$ satisfies acc on annihilators.
\end{prop}

\begin{proof}
Since $M$ is $M^{(\Lambda)}$-projective for every index set $\Lambda$, there is a surjetive ring homomorphism $\rho:\End_R(M)\to \End_R(\overline{M})$. Let $\overline{X}$ be any subset of $\overline{M}$ and set $\overline{S}=\End_R(\overline{M})$. We claim that the inverse image under $\rho$ of $\mathbf{l}_{\overline{S}}(\overline{X})$ is $\mathbf{l}_S(N_M\left\langle X\right\rangle)$ where $X$ is the inverse image of $\overline{X}$ under the canonical projection $\pi:M\to \overline{M}$. For each $\bar{f}\in\mathbf{l}_{\overline{S}}(\overline{X})$ there is an $f\in\End_R(M)$ such that $f(X)\subseteq \ann^r_M(N)$. Note that $f(\left\langle X\right\rangle)=\left\langle f(X)\right\rangle \leq \ann^r_M(N)$. By Lemma \ref{fprod}, $f(N_M\left\langle X\right\rangle )=N_Mf(\left\langle X\right\rangle )=0$. Therefore, $f\in\mathbf{l}_S(N_M\left\langle X\right\rangle )$. This implies that $\rho^{-1}(\mathbf{l}_{\overline{S}}(\overline{X}))\subseteq\mathbf{l}_S(N_M\left\langle X\right\rangle )$. Now, let $g\in\mathbf{l}_S(N_M\left\langle X\right\rangle )$. Again by Lemma \ref{fprod}, $0=g(N_M\left\langle X\right\rangle )=N_Mg(\left\langle X\right\rangle )$. Hence $g(\left\langle X\right\rangle )\leq \ann^r_M(N)$. Therefore, $\rho(g)(\overline{X})=\rho(g)\pi(X)=\pi g(X)=0$. Thus, $\mathbf{l}_S(N_M\left\langle X\right\rangle )\subseteq\rho^{-1}(\mathbf{l}_{\overline{S}}(\overline{X}))$, proving the claim. It follows from Lemma \ref{dccl} that $\overline{M}$ satisfies acc on annihilators.
\end{proof}

Let $M$ be a module and $K\leq N\leq M$. Let $\mathbf{l}_N(K)$ denote the intersection $\ann_M(K)\cap N$. On the other hand, let $\mathbf{r}_N(K)=\ann^r_M(K)\cap N$. If $N$ is a fully invariant submodule of $M$ and $K\leq N$, then $\mathbf{l}_N(K)$ and $\mathbf{r}_N(K)$ are fully invariant in $M$.

\begin{cor}\label{dccrn}
Let $M$ be a module. If $M$ satisfies acc on annihilators then the set $\{\mathbf{l}_N(K)\mid K\leq N\}$ satisfies acc and the set $\{\mathbf{r}_N(K)\mid K\leq N\}$ satisfies dcc.
\end{cor}

\begin{proof}
Let $K\leq N\leq M$. We claim that $\mathbf{l}_N(K)=\ann_M(\ann^r_M(\mathbf{l}_N(K)))\cap N$. Since $\mathbf{l}_N(K)_M\ann^r_M(\mathbf{l}_N(K))=0$, we have that $\mathbf{l}_N(K)\subseteq \ann_M(\ann^r_M(\mathbf{l}_N(K)))\cap N$. On the other hand, $\mathbf{l}_N(K)\subseteq \ann_M(K)$. Hence $\ann^r_M(\mathbf{l}_N(K))\supseteq \ann^r_M(\ann_M(K))$. Therefore, $\ann_M(\ann^r_M(\mathbf{l}_N(K)))\subseteq \ann_M(\ann^r_M(\ann_M(K)))=\ann_M(K)$. Thus, $\ann_M(\ann^r_M(\mathbf{l}_N(K)))\cap N\subseteq \ann_M(K)\cap N=\mathbf{l}_N(K)$ proving the claim. Now, if 
\[\mathbf{l}_N(K_1)\subseteq\mathbf{l}_N(K_2)\subseteq\mathbf{l}_N(K_3)\subseteq\cdots\]
is an ascending chain, applying $\ann^r(-)$ to the chain, we get a descending chain
\[\ann^r_M(\mathbf{l}_N(K_1))\supseteq \ann^r_M(\mathbf{l}_N(K_2))\supseteq \ann^r_M(\mathbf{l}_N(K_3))\supseteq\cdots\]
By Lemma \ref{dccannr}, there exists $\ell>0$ such that $\ann^r_M(\mathbf{l}_N(K_\ell))= \ann^r_M(\mathbf{l}_N(K_{\ell+i}))$ for all $i\geq 0$. It follows that
\[\mathbf{l}_N(K_\ell)=\ann_M(\ann^r_M(\mathbf{l}_N(K_\ell)))\cap N = \ann_M(\ann^r_M(\mathbf{l}_N(K_{\ell+i})))\cap N=\mathbf{l}_N(K_{\ell+i})\]
for all $i\geq 0$. Analogously, the set $\{\mathbf{r}_N(K)\mid K\leq N\}$ satisfies dcc.
\end{proof}

\begin{lemma}\label{rannncero}
Suppose $M$ is $M^{(\Lambda)}$-projective for every index set $\Lambda$ and let $N< M$ be a fully invariant nil-submodule. If $M$ 
satisfies acc on annihilators, then $\mathbf{r}_N(N)\neq 0$.
\end{lemma}

\begin{proof}
Set $\Gamma=\{\mathbf{r}_N(K)\mid K\leq N\text{ and } K \text{ finitely generated}\}$. It  follows from Corollary \ref{dccrn} that $\Gamma$ has minimal elements. Let $\mathbf{r}_N(K)$ be a minimal element in $\Gamma$ and let $n\in N$. Then $K+Rn$ is a finitely generated submodule of $N$ and so $\mathbf{r}_N(K)=\mathbf{r}_N(K+Rn)\subseteq\mathbf{r}_N(Rn)$. Hence, $\mathbf{r}_N(K)\subseteq\bigcap_{n\in N}\mathbf{r}_N(Rn)=\bigcap_{n\in N}\ann^r_M(Rn)\cap N=\ann^r_M(\sum_{n\in N}Rn)\cap N=\ann^r_M(N)\cap N=\mathbf{r}_N(N)$ by Lemma \ref{rannintersection}. 
Now, by Lemma \ref{fgnilp} and Proposition \ref{accnillocnil}, $K$ is nilpotent, that is, there exists $\ell>1$ such that $K^\ell=0$ but $K^{\ell-1}\neq 0$. Therefore $0\neq K^{\ell-1}\subseteq\mathbf{r}_N(K)$.
\end{proof}

\begin{prop}\label{subm}
Suppose $M$ is $M^{(\Lambda)}$-projective for every index set $\Lambda$ and let $N<M$ be a fully invariant nil-submodule such that $\mathbf{l}_N(N^j)=0$ for all $j>0$. If $M$ 
satisfies acc on annihilators, then there exist submodules $A_1,...,A_k,...\subseteq N$ such that:
\begin{enumerate}
\item ${A_1}_M\cdots_M{A_k}\neq 0$ for all $k>0$; and
\item ${A_1}_M\cdots_M{A_k}_MA_n=0$ if $n\leq k$.
\end{enumerate}
\end{prop}

\begin{proof}
We construct the submodules inductively. Let $A_1=\mathbf{r}_N(N)\neq 0$ by Lemma \ref{rannncero}. Note that ${A_1}_M{A_1}=0$. To get $A_2$, consider $\ann^r_M(A_1)$. Then $\frac{N+\ann^r_M(A_1)}{\ann^r_M(A_1)}$ is a nil-submodule of $M/\ann^r_M(A_1)$ by Lemma \ref{factornil}. Moreover, $M/\ann^r_M(A_1)$ 
satisfies acc on annihilators by 
Proposition \ref{factorrightacc}. Write $\overline{N}=\frac{N+\ann^r_M(A_1)}{\ann^r_M(A_1)}$ and $\overline{M}=M/\ann^r_M(A_1)$. By Lemma \ref{rannncero}, $\mathbf{r}_{\overline{N}}(\overline{N})\neq 0$. Hence there exists $0\neq T/\ann^r_M(A_1)\leq \overline{M}$ such that $\overline{N}_{\overline{M}}\left( T/\ann^r_M(A_1)\right) =0$. This implies that $N_MT\subseteq \ann^r_M(A_1)$ and so ${A_1}_M{N}_MT=0$ but ${A_1}_MT\neq 0$. Let $A_2=\mathbf{r}_N({A_1}_MN)$. Therefore ${A_1}_M{A_2}\neq 0$. We have that ${A_1}_M{A_2}\subseteq N$, then ${A_1}_M{A_2}_M{A_1}=0$. On the other hand, ${A_1}_M{A_2}\subseteq {A_1}_MN$. Hence, ${A_1}_M{A_2}_M{A_2}\subseteq {A_1}_MN_M{A_2}=0$ because definition of $A_2$. Inductively, each $A_{i+1}=\mathbf{r}_N({A_1}_M\cdots_M{A_i}_MN)$.
\end{proof}

\section{Nilpotency of the prime radical of a Goldie module}\label{Nilpotency}

In this section, the main theorem is proved and as corollary we get that the prime radical of a Goldie module $M$ is nilpotent. In \cite[Theorem 1]{lanski1969nil} Lanski proved that nil subrings of a left Goldie ring are nilpotent. We will make use of this result applied to the endomorphism ring of a Goldie module. For, we start with the following lemmas.

\begin{lemma}\label{accmoduloann}
	Let $M$ be a quasi-projective module with endomorphism ring $S=\End_R(M)$ and $N=\bigcap_{f\in I}\Ker f$ for some ideal $I$ of $S$. If $M$ satisfies acc on annihilators then so does $M/N$. 
\end{lemma}

\begin{proof}
	Note that $N$ is a fully invariant submodule of $M$ because $I$ is an ideal. It is enough to show that inverse image, under the canonical projection $\pi:M\to M/N$, of an annihilator in $M/N$ is an annihilator in $M$. Let $\overline{Y}\subseteq\End_R(M/N)$. Consider $\bigcap_{\bar{f}\in\overline{Y}}\Ker \bar{f}$ in $M/N$. Hence $\bigcap_{\bar{f}\in\overline{Y}}\Ker \bar{f}=A/N$ for some submodule $A$ of $M$. Since $M$ is quasi-projective, given $\bar{f}\in\overline{Y}$ there exists $f\in\End_R(M)$ such that $\pi f=\bar{f}\pi$. Let $\widehat{Y}=\{f\in\End_R(M)\mid \pi f=\bar{f}\pi\text{ for some }\bar{f}\in\overline{Y}\}$. It follows that $f(A)\subseteq N$ for all $f\in \widehat{Y}$ and so, $gf(A)=0$ for all $f\in \widehat{Y}$ and all $g\in I$. Let $Y=\{gf\mid g\in I\;\text{and}\;f\in \widehat{Y}\}$. Therefore, $A\subseteq\bigcap_{gf\in Y}\Ker gf$. Now, let $x\in\bigcap_{gf\in Y}\Ker gf$ and let $\bar{f}\in\overline{Y}$. Then $\bar{f}\pi(x)=\pi(f(x))$. Since $g(f(x))=0$ for all $g\in I$, $f(x)\in N$. Hence $\pi(f(x))=0$. It follows that $x\in A$. Thus $A=\bigcap_{gf\in Y}\Ker gf$.
\end{proof}

Recall that a module $M$ is said to be \emph{retractable} if $\Hom_R(M,N)\neq 0$ for every nonzero submodule $N$ of $M$.

\begin{lemma}\label{fmret}
	Let $M$ be $M^{(\Lambda)}$-projective for every index set $\Lambda$ and let $N\leq M$. If $M$ is retractable, then 
	\begin{enumerate}
		\item $M/\ann_M^r(N)$ is retractable.
		\item $M/\ann_M(N)$ is retractable. 
	\end{enumerate}
\end{lemma}

\begin{proof}
	(1) Let $K/\ann_M^r(N)$ be a submodule of $M/\ann_M^r(N)$ and suppose that 
	\[\Hom_R(M/\ann_M^r(N),K/\ann_M^r(N))=0.\]
	Since $\ann_M^r(N)$ is fully invariant in $M$, any homomorphism $f:M\to K$ induces a homomorphism $\overline{f}:M/\ann_M^r(N)\to K/\ann_M^r(N)$ such that $\pi f=\overline{f}\pi$ where $\pi:M\to M/\ann_M^r(N)$ is the canonical projection. This implies that for all $f\in\Hom_R(M,K)$, $f(M)\subseteq \ann_M^r(N)$. Therefore $0=N_M(M_MK)=(N_MM)_MK=N_MK$. Hence $K\subseteq\ann_M^r(N)$ and so $K/\ann_M^r(N)=0$. Thus $M/\ann_M^r(N)$ is retractable.
	
	(2) Let $K/\ann_M(N)$ be a submodule of $M/\ann_M(N)$ and suppose that 
	\[\Hom_R(M/\ann_M(N),K/\ann_M(N))=0.\]
	As in the previous proof, we get $0=tr^M(K)_MN=(M_MK)_MN=M_M(K_MN)=tr^M(K_MN)$. This implies that $K_MN=0$ because $M$ is retractable. Thus, $K\subseteq \ann_M(N)$. Hence $M/\ann_M(N)$ is retractable.
\end{proof}

\begin{lemma}\label{mgolsgol}
Suppose $M$ is $M^{(\Lambda)}$-projective for every index set $\Lambda$ and retractable. If $M$ is Goldie then $S=\End_R(M)$ is a right Goldie ring.
\end{lemma}

\begin{proof}
By Proposition \ref{maccsacc}, $S$ satisfies acc on right annihilators. Since $M$ has finite uniform dimension, so does $S_S$ by \cite[Theorem 2.6]{HaghanyStudy}. Thus, $S$ is a right Goldie ring.
\end{proof}

\begin{thm}\label{main}
Suppose $M$ is $M^{(\Lambda)}$-projective for every index set $\Lambda$ and retractable. If $M$ is a Goldie module and $N<M$ is a fully invariant nil-submodule, then $N$ is nilpotent. 
\end{thm}

\begin{proof}
There is an ascending chain of submodules of $M$,
\[\ann_M (N)\subseteq \ann_M (N^2)\subseteq \ann_M (N^3)\subseteq\cdots\]
By hypothesis, there exists a positive integer $k$ such that $\ann_M (N^k)=\ann_M (N^{k+i})$ for all $i>0$. Denote $\overline {N}=(N+\ann_M(N^k))/\ann_M (N^k)$ and $\overline {M}=M/\ann_M(N^k)$.  We claim that $\mathbf{l}_{\overline {N}}(\overline {N}^j)=0$ for all $j>0$. 
If $\mathbf{l}_{\overline {N}}(\overline {N}^j)=A/\ann_M(N^k)$ with $A\leq M$, then $A_MN^j\leq \ann_M(N^k)$. Hence, $A_MN^{k+j}=(A_MN^j)_MN^{k}=0$. Therefore $A\subseteq \ann_M(N^{k+j})=\ann_M(N^k)$. Thus $\mathbf{l}_{\overline {N}}(\overline {N}^j)=0$, proving the claim.

If $N$ is not nilpotent, then $\overline {N}$ is a nil-submodule of $\overline {M}$ by Lemma \ref{factornil}. Moreover, $\overline {M}$ is retractable and satisfies acc on annihilators by Lemma \ref{fmret} and Lemma \ref{accmoduloann}. Hence, there exist nonzero submodules $\overline{A_1},...,\overline{A_i},...\subseteq \overline{N}$ by Proposition \ref{subm} such that 
\begin{itemize}
\item[(i)] ${\overline{A_1}}_{\overline{M}}\cdots_{\overline{M}}{\overline{A_i}}\neq 0$ for all $i>0$; and
\item[(ii)] ${\overline{A_1}}_{\overline{M}}\cdots_{\overline{M}}{\overline{A_i}}_{\overline{M}}\overline{A_j}=0$ if $j\leq i$.
\end{itemize}
Consider the inverse images of the submodules $\overline{A_i}$, $\ann_M(N^{k})\subseteq
A_1,..,A_i,...\subseteq N+\ann_M(N^k)$. Intersecting with $N$, we have
\[\mathbf{l}_N(N)\subseteq \ann_M(N^k)\cap N \subseteq (A_1\cap N),...,(A_i\cap N),...\subseteq N.\]
Suppose that $\mathbf{l}_N(N)\neq 0$. Then $A_i\cap N\neq 0$. Let $B_i$ denote the submodule $A_i\cap N$. Hence, the submodules $B_1,...,B_i,...$ satisfy:  
\begin{enumerate}
\item ${B_1}_M\cdots_M{B_i}\neq 0$ for all $i>0$; and
\item ${B_1}_M\cdots_M{B_i}_M{B_j}\subseteq \mathbf{l}_N(N^k)$ if $j\leq i$.
\end{enumerate}
Condition (1) implies that there exists a sequence $\{f_1,f_2,...,f_i,...\}\subseteq\Hom_R(M,N)$, with $f_i:M\to B_i$ for all $i>0$, such that $f_i\cdots f_2f_1\neq 0$ for all $i>0$ and the condition (2) says that $g(f_jf_i\cdots f_2f_1(M))=0$ for all $g:M\to N^k$ if $j\leq i$. Note that $f_{i+1}f_i(M)\subseteq f_{i+1}(N)\subseteq N^2$. Consider $\{f_{j+1},...,f_{j+k}\}$ with $j\leq i$. Then $f_{j+k}\cdots f_{j+1}(M)\subseteq N^k$, that is, $f_{j+k}\cdots f_{j+1}:M\to N^k$. It follows that $f_{j+k}\cdots f_{j+1}f_jf_i\cdots f_2f_1=0$ for $j\leq i$. Thus, we just got the elements described in \cite[Lemma 8]{lanski1969nil}, that is, elements $\{f_1,f_2,...,f_i,...\}\subseteq S$ such that 
\begin{enumerate}
\item $f_i\cdots f_2f_1\neq 0$ for all $i>0$; and
\item $f_{j+k}\cdots f_{j+1}f_jf_i\cdots f_2f_1=0$ for $j\leq i$.
\end{enumerate}
Since $M$ is a Goldie module, if follows from Lemma \ref{mgolsgol} that $S$ is a right Goldie ring. Following the proof of \cite[Theorem 1]{lanski1969nil}, we arrive to a contradiction and so $N$ is nilpotent. 

In the case $\mathbf{l}_N(N)=0$, then we start the proof with $M=\overline{M}$ and $A_i=\overline{A_i}$ for all $i>0$.
\end{proof}

\begin{cor}\label{primenilgoldie}
Suppose $M$ is $M^{(\Lambda)}$-projective for every index set $\Lambda$ and retractable. If $M$ is a Goldie module, then the prime radical of $M$ is nilpotent.
\end{cor}



\bibliography{sjnames,biblio}
\bibliographystyle{acm} 


\end{document}